\documentclass[a4paper,11pt]{article}
\usepackage{epsfig,graphicx}
\usepackage{amsmath,amsfonts,verbatim}
\usepackage{tikz}

\def\diag{{\mbox{\rm diag}}}
\def\spec{{\mbox{\rm sp\,}}}
\def\suma{{\mbox{\rm sum}}}

\def\Mat{{\mbox{\rm Mat}}}
\def\span{{\mbox{\rm span}}}
\def\trace{{\mbox{\rm tr}}}
\def\gzit#1{{\rm (\ref{#1})}}
\def\0{{\mbox{\boldmath{$0$}}}}

\def\j{\pmb j}

\newcommand{\R}{{\mathbb R}}

\newcommand{\qed}{\hfill\hbox{\rule{3pt}{6pt}}}

\newenvironment{proof}{\noindent{\sc Proof. }}{\nopagebreak\hspace*{0.5cm}\hfill$\qed$\vspace{0.5cm}}

\newcommand{\G}{\Gamma}

\newtheorem{theorem}{Theorem}[section]
\newtheorem{lemma}[theorem]{Lemma}
\newtheorem{corollary}[theorem]{Corollary}
\newtheorem{proposition}[theorem]{Proposition}

\newtheorem{researchProblem}[theorem]{Research problem}
\title{On a version of the spectral excess theorem\footnote{This research has been partially supported by 
		AGAUR from the Catalan Government under project 2017SGR1087 and by MICINN from the Spanish Government under project PGC2018-095471-B-I00. The second author acknowledges the financial support from the Slovenian Research Agency (research core funding No. P1-0285 and Young Researchers Grant).}}

\author{M. A. Fiol\\
	{\small Departament de Matem\`atiques} \\
   	{\small     Universitat Polit\' ecnica de Catalunya}\\
    	{\small    Barcelona Graduate School of Mathematics} \\
      	{\small  Catalonia, Spain} \\
       	{\small {\tt miguel.angel.fiol@upc.edu}} \and
Safet Penji\'c\\
	{\small Andrej Maru\v{s}i\v{c} Institute}\\
      	{\small  University of Primorska} \\
     	{\small  Muzejski trg 2 }\\
    	{\small    6000 Koper, Slovenia} \\
     	{\small   {\tt safet.penjic@iam.upr.si}}}

\begin{document}

\maketitle

\begin{abstract}
Given a regular (connected) graph $\G=(X,E)$ with adjacency matrix $A$, $d+1$ distinct eigenvalues, and diameter $D$, we give a characterization  of when its distance matrix $A_D$ is a polynomial in $A$, in terms of the adjacency spectrum of $\G$ and the arithmetic (or harmonic) mean of the numbers of vertices at distance $\le D-1$ of every vertex.
The same results is proved for any graph by using its Laplacian matrix $L$ and corresponding spectrum.
When $D=d$ we reobtain the spectral excess theorem characterizing distance-regular graphs.
\end{abstract}


\smallskip
{\small


}

\section{Preliminaries}
\label{prelim}
Let $\G=(X,E)$ be a (simple and connected) graph on $n=|X|$ vertices, with adjacency matrix $A$, and spectrum $\spec \G=\{\lambda_0^{m(\lambda_0)},\lambda_1^{m(\lambda_1)},\ldots,\lambda_d^{m(\lambda_d)}\}$, where $\lambda_0>\lambda_1>\cdots>\lambda_d$ are the distinct eigenvalues, and the superscripts stand for their multiplicities $m_i=m(\lambda_i)$. If $\G$ has diameter $D$, we denote by $\G_i(x)$ the set of vertices at distance $i=1,\ldots,D$ from $x\in X$, and $k_i(x)=|\G_i(x)|$. We abbreviate $k_1(x)$ by $k(x)$, the degree of vertex $x$. 

Given to square matrices $M,N\in \Mat_{n\times n}(\R)$, let $\suma(M)$ denote the sum of all entries of $M$, so that $\trace(MN)=\suma(M\circ N)$, where `$\circ$' stand for the Hadamard (or entrywise) product.
The {\em predistance polynomials} $p_0,p_1,\ldots,p_d$ of $\G$, introduced in \cite{fg97}, are a sequence of orthogonal polynomials with respect to the scalar product
$$
\langle f,g\rangle_{A}=\frac{1}{n}\trace(f(A)g(A))=\frac{1}{n}\suma(f(A)\circ g(A))=\frac{1}{n}\sum_{i=0}^d m(\lambda_i)f(\lambda_i)g(\lambda_i),
$$
normalized in such a way that $\|p_i\|^2_{A}=p_i(\lambda_0)$. For instance, since $\trace (A^h)=n\langle A^h,I \rangle_A=\sum_{i=0}^d m(\lambda_i)\lambda_i^h$, the two first predistance polynomials are $p_0(x)=1$ and $p_1(x)=\frac{\lambda_0}{\overline{k}}x$, with $\overline{k}$ being the average degree of $\G$, see also Lemma \ref{r1}. (It is known that $\overline{k}\le \lambda_0$ with equality if and only if $\G$ is regular.)
Moreover, the value of the highest degree polynomial $p_d$ at $\lambda_0$ can be computed from $\spec \G$ as 
\begin{equation}
\label{pd(lambda0)}
p_d(\lambda_0)=n\left(\sum_{i=0}^d \frac{\phi_0^2}{m_i\phi_i^2}\right)^{-1},
\end{equation}
where $\phi_i =\prod_{j\neq i}(\lambda_i-\lambda_j)$, $i=0,\ldots,d$ (see \cite{fg97}).

The {\em predistance matrices} $P_0,P_1,\ldots,P_d$ are then defined by $P_i=p_i(A)$ for $i=0,1,\ldots,d$. By  \cite[Prop. 2.2]{cffg09}, there exist numbers $\alpha_i$, $\beta_i$, and $\gamma_i$  such that $PP_i(=P_iP)=\beta_{i-1}P_{i-1}+\alpha_iP_i+\gamma_{i+1}P_{i+1}$ for $i=0,\ldots,d$, where $P=P_1$, $P_{-1}=P_{d+1}=0$, $\gamma_1=1$, and $\alpha_0=0$. Also, if we define the polynomial $p_{d+1}(x)=(x-\alpha_d)p_d(x)-\beta_{d-1}p_{d-1}(x)$, it can be shown that $p_{d+1}(A)=0$, and the distinct eigenvalues of $\G$ are precisely the zeros of $p_{d+1}$.

The above names come from the fact that, if $\G$ is a distance-regular graph, then the $p_i$'s and $P_i$'s correspond to the well-known distance polynomials and distance matrices $A_i$, respectively. In fact, a known characterization states that $\G$ is distance-regular if and only if such polynomials satisfy $p_i(A)=A_i$ for every $i=1,\ldots,D$. Moreover, in this case, $D=d$.
If we do not impose that the degree of each polynomial coincide with its subindex, then it can be $D<d$ and the graph is called {\em distance-polynomial}, a concept introduced by Weichsel \cite{we82}.

In fact, if $D=d$, the first author, Garriga, and Yebra \cite{fgy96} proved the following.
\begin{proposition}
\label{charac-1}
A regular graph $\G$ with diameter $D$ and $d+1$ distinct eigenvalues is distance-regular if and only if $D=d$ and its highest degree predistance polynomial satisfies $p_d(A)=A_d$.
\end{proposition}
  
From the predistance polynomials, we also consider their sums $q_i=p_0+\cdots +p_i$ for $i=0,\ldots,d$, which satisfy
$1=q_0(\lambda_0)<q_1(\lambda_0)<\cdots <q_d(\lambda_0)=|X|$, with $q_d=H$ being the Hoffman polynomial that characterizes the regularity of $\G$ by the equality $H(A)=J$, the all-1 matrix (see \cite{AJH}).

We also recall that the Laplacian matrix of $\G$ is the matrix $L=K-A$, where $K=\diag(k(x_1),\ldots,k(x_n))$, where $x_i\in X$ for $i=1,\ldots,n$. The Laplacian spectrum of $\G$ is $\spec_{\!L} \G=\spec L$  $=\{\theta_0^{m(\theta_0)},\theta_1^{m(\theta_1)},\ldots,\theta_d^{m(\theta_d)}\}$ with $\theta_0=0<\theta_1<\cdots <\theta_d$. In particular, since $\G$ is connected, $m_0=1$, and the eigenvalue $0$ has eigenvector $\j$, the all-1 vector.
As in the case of the adjacency spectrum, we can define the {\em Laplacian predistance polynomials} $r_0,r_1,\ldots,r_d$ as the sequence of orthogonal polynomials with respect to the scalar product
$$
\langle f,g\rangle_{L}=\frac{1}{n}\trace(f(L)g(L))=\frac{1}{n}\suma(f(L)\circ g(L))=\frac{1}{n}\sum_{i=0}^d m(\theta_i) f(\theta_i)g(\theta_i),
$$
normalized in such a way that $\|r_i\|_L^2=r_i(0)$.
The following result gives
the first two Laplacian predistance polynomials.
\begin{lemma} 
	\label{r1}
	Let $\G$ be a graph with Laplacian matrix $L=K-A$.
	Let  $\overline{k^2}$ be the  average of the square degrees of $\G$.
	Then
	\begin{itemize}
		\item[$(i)$] 
		$r_0(x)=1$.
		\item[$(ii)$]
		$r_1(x)=\frac{\overline{k}}{\overline{k}(\overline{k}-1)-\overline{k^2}}(x-\overline{k})$. 
	\end{itemize}
	\end{lemma}
\begin{proof}
	We only need to prove $(ii)$. By using the method of Gram-Schmidt, 
	we first find a polynomial $t(x)$ orthogonal to $r_0=1$. That is,
	$t(x)=x-\frac{\langle x, r_0\rangle_L}{\|r_0\|_L^2}r_0(x)$, where
$$
\langle x, r_0\rangle_L =\frac{1}{n}\trace(L)=\frac{1}{n}\sum_{x\in X} k(x)=\overline{k} \qquad \mbox{and} \qquad \|r_0\|_L^2=\frac{1}{n}\trace(I^2)=1.
$$
Now, $r_1(x)=\alpha t(x)$, where $\alpha$ is a constant to be determined by the normalization condition $\|r_1\|_L=r_1(0)$, which gives $\alpha=\frac{t(0)}{\|t\|_L^2}$. Moreover,
\begin{eqnarray*}
\|t\|_L^2 & = &\|x-\overline{k}\|_L^2 = \frac{1}{n}\trace([K-A-\overline{k}I]^2)
=\frac{1}{n}\trace(K^2+A^2+(\overline{k})^2I-2KA-2\overline{k}K+2\overline{k}A)\\
 & = & \frac{1}{n}\sum_{x\in X}k(x)^2+\frac{1}{n}\sum_{x\in X}k(x)
+ (\overline{k})^2-2\overline{k}\frac{1}{n}\sum_{x\in X}k(x)=\overline{k^2}+\overline{k}+(\overline{k})^2-2(\overline{k})^2\\
 & = & \overline{k^2}+\overline{k}-(\overline{k})^2.
\end{eqnarray*}
Then, from $r_1(x)=\frac{t(0)}{\|t\|^2_L}t(x)$ we get the result. 
\end{proof}

Also, as in the case of the predistance polynomials $p_i$'s, we have
\begin{equation}
\label{rd(0)}
r_d(0)=n\left(\sum_{i=0}^d \frac{\psi_0^2}{m(\theta_i)\psi_i^2}\right)^{-1},
\end{equation}
with $\psi_i =\prod_{j\neq i}(\theta_i-\theta_j)$, $i=0,\ldots,d$ (see \cite{cffg09}).

The analogous of Proposition \ref{charac-1}, for not necessarily regular graphs, was proved by Van Dam and the first author in \cite{vf14}.

\begin{proposition}
	\label{charac-2}
	A graph $\G$ with Laplacian matrix $L$, $d+1$ distinct Laplacian eigenvalues, and  diameter $D$ is distance-regular if and only if $D=d$ and its highest degree Laplacian predistance polynomial satisfies $r_d(L)=A_d$.
\end{proposition}

In fact, the regularity of $\G$ is already implied by the equation $r_1(L)=A$, as shown in the following lemma.

\begin{lemma}
\label{r1-reg}
Let $\G$ be a graph with adjacency and Laplacian matrices $A$ and $L$, respectively, and Laplacian predistance polynomial $r_1$.
Then, $\G$ is $k$-regular if and only if $r_1(L)=A$.
\end{lemma}
\begin{proof}
From the Cauchy-Schwartz inequality,
$\overline{k^2}\ge (\overline{k})^2$, with equality if and only if $\G$ is $k$-regular. In this case, Lemma \ref{r1}$(ii)$ becomes $r_1(x)=k-x$ and, hence, $r_1(L)=kI-(kI-A)=A$. Conversely, if $r_1(L)=r_1(K-A)=A$, by equating the coefficients of $A$ in  \eqref{r1}, we get $\frac{\overline{k}}{\overline{k}(\overline{k}-1)-\overline{k^2}}=-1$, whence $\overline{k^2}=(\overline{k})^2$, $\G$ is $k$-regular, and $K=kI$.
\end{proof}

In this context, we also consider the sum polynomials $s_i=r_0+\cdots +r_i$ for $i=0,\ldots,d$, with $H_L=s_d$ being a Hoffman-like polynomial satisfying $H(L)=J$ (independently of whether $\G$ is regular or not). For more details, see \cite{vf14}.

In our results we use the following simple result.

\begin{lemma}
\label{1e}
Let $\G=(X,E)$ be a graph with adjacency matrix $A$ and Laplacian matrix $L$. Given a vertex $x\in X$ and a polynomial $p\in\R_h[t]$,
\begin{itemize}
\item [$(i)$]
If $\G$ is $k$-regular, then $\sum_{y\in X} p(A)_{xy}=p(k)$.
\item[$(ii)$]
If $\G$ is a general graph, then $\sum_{y\in X} p(L)_{xy}=p(0)$.
\end{itemize}  
\end{lemma}
\begin{proof}
$(i)$ Since $\G$ is $k$-regular, $(k,\j)$ is an eigenpair of $A$ and, hence,
$p(A)\j=p(k)\j$. Then, the result follows by considering the $x^{\rm th}$ component of both vectors.
Case $(ii)$ is proved in the same way by considering that $(0,\j)$ is an eigenpair of $L$.
\end{proof}

\section{A version of the spectral excess theorem}
\label{ae}
The spectral excess theorem, due to Fiol and  Garriga \cite{fg97}, states that a regular (connected) graph $\G$ is distance-regular if and only if its spectral excess (a number which can be computed from the spectrum of $\G$) equals its average excess (the mean of the numbers of vertices at maximum distance from every vertex), see  Van Dam \cite{vd08}, and Fiol, Gago, Garriga \cite{fgg10} for short proofs.

In this section we find a possible blue solution to the problem of deciding  whether, from the adjacency spectrum of a (regular) graph $\G$ and the  harmonic (or arithmetic) mean of the numbers $(|X|-|\G_D(x)|)_{x\in X}$, we can decide that $A_D$ is a polynomial in $A$.
To be more precise, we provide a characterization of when $A_D\in \span \{p_0(A),\ldots,p_d(A)\}$, where the $p_i's$ are the predistance polynomials.


Before proving the main result, note that,
for any $x\in X$ and any $C\in\Mat_X(\R)$,  the Cauchy-Schwartz inequality yields
$$
\left(
\sum_{y\not\in\G_{D}(x)}(C_{xy})^2
\right)
\left(
\sum_{y\not\in\G_{D}(x)}1^2
\right)
\ge
\left(
\sum_{y\not\in\G_{D}(x)}C_{xy}
\right)^2.
$$
That is,
\begin{equation}
\label{1d}
\sum_{y\not\in\G_{D}(x)}(C_{xy})^2
\ge
\frac{1}{|X|-|\G_D(x)|}
\left(
\sum_{y\not\in\G_{D}(x)}C_{xy}
\right)^2,
\end{equation}
and equality holds if and only if all the values of $C_{xy}$ are the same for all $y\not\in\G_{D}(x)$.

\begin{theorem}
\label{th:1f}
Let $\G=(X,R)$ be a connected $k$-regular graph with $d+1$ distinct eigenvalues, diameter $D$, and predistance polynomials $\{p_i\}_{i=0}^d$. Then,
\begin{equation}
\label{1f}
\frac{|X|}{\sum_{x\in X}\frac{1}{|X|-|\G_D(x)|}}\ge q_{D-1}(k)=|X| - \sum_{i=D}^d p_i(k),
\end{equation}
with equality if and only if
$A_D=\sum_{i=D}^d p_i(A)$.
\end{theorem}

\noindent
\begin{proof}
We just adjust the proof of \cite[Lemma 1]{vd08}, together with Lemma \ref{1e} and \gzit{1d}.
Recalling that $q_{D-1}=\sum_{i=0}^{D-1} p_i$, we have that
\begin{eqnarray*}
q_{D-1}(k) & = & \langle q_{D-1},q_{D-1}\rangle_A=\frac{1}{|X|}\trace (q_{D-1}(A)^2)
=\frac{1}{|X|}\sum_{x\in X} (q_{D-1}(A)^2)_{xx}\\
 & = &
\frac{1}{|X|}\sum_{x\in X}\sum_{y\not\in \G_D(x)} (q_{D-1}(A)_{xy})^2
\ge
\frac{1}{|X|}\sum_{x\in X}
\frac{1}{|X|-|\G_D(x)|}
\left[
\sum_{y\not\in\G_{D}(x)} q_{D-1}(A)_{xy}
\right]^2\\
 & = &
\frac{1}{|X|}\sum_{x\in X}
\frac{1}{|X|-|\G_D(x)|}
\left[
q_{D-1}(k)
\right]^2,
\end{eqnarray*}
and this yields
$$
\frac{|X|}{\sum_{x\in X}\frac{1}{|X|-|\G_D(x)|}}\ge
q_{D-1}(k).
$$
Since $|X|=\sum_{i=0}^d p_i(k)=q_{D-1}(k)+\sum_{i=D}^d p_i(k)$, the inequality follows.

If equality holds, then for each $x$ the values of $q_{D-1}(A)_{xy}$ are the same, say $\alpha$, for all $y\not\in\G_D(x)$. Moreover, since $A$ is symmetric, $q_{D-1}(A)_{xy}=q_{D-1}(A)_{yx}$ for any $x$ and $y\not\in\G_D(x)$, so that $q_{D-1}(A)_{xy}=\alpha$ for any $x$ and $y\not\in\G_D(x)$. Also, 
by Lemma \ref{1e}$(i)$, 
 $\sum_{y\in X} q_{D-1}(A)_{xy}= q_{D-1}(k)=n_{D-1}\alpha$, where $n_{D-1}=\sum_{i=0}^{D-1}|\G_i(x)|$ for any $\in X$. Finally, from  $\|q_{D-1}\|^2=q_{D-1}(k)$, we have
 $$
 \|q_{D-1}\|^2=\frac{1}{n}\trace (q_{D-1}(A)^2)=\frac{1}{n}\alpha^2 \suma \left(\sum_{i=0}^{D-1}A_i\right)=n_{D-1}\alpha^2=n_{D-1}\alpha,
 $$
so that $\alpha=1$.
That is, $q_{D-1}(A)_{xy}=1$ for each pair of vertices $x$ and $y$ at distance less than $D-1$. Consequently, $q_{D-1}(A)=\sum_{i=0}^{D-1}A_i=J-A_D$, which yields $A_D=\sum_{i=D}^d p_i(A)$.

Conversely, assume that $\sum_{i=D}^d p_i(A)=A_D$. Then, $A_D\j=\sum_{i=D}^d p_i(k)\j$, and with $|\G_D(x)|=\sum_{i=D}^d p_i(k)$, the equality follows.
\end{proof}

As a simple consequence, notice that, if $\G$ is a $k$-regular graph of diameter $2$, then $|X|-|\G_2(x)|=1+k$ for any $x\in X$. Besides, $q_1(k)=p_0(k)+p_1(k)=1+k$. Thus, equality in Theorem \ref{th:1f} holds, and $\G$ is distance polynomial, as already proved Weichel in \cite{we82}.
Another consequence of Theorem \ref{th:1f} is the following corollary.

\begin{corollary}
\label{1g}
Let $\G=(X,R)$ be a connected $k$-regular graph on $n$ vertices, with spectrum $\spec \G = $ $\{\lambda_0(=k)^{m(\lambda_0)},\lambda_1^{m(\lambda_1)}\ldots,\lambda_d^{m(\lambda_d)}\}$, diameter $D$, and predistance polynomials $\{p_i\}_{i=0}^d$. Then the following holds.
\begin{itemize}
\item[$(i)$]
In general,
$$
\frac{1}{|X|}\sum_{x\in X}(|X|-|\G_D(x)|)\ge
|X| - \sum_{i=D}^d p_i(k)
$$
with equality if and only if $A_D=\sum_{i=D}^d p_i(A)$.
\item[$(ii)$] If $A_D\in\span\{I,A,\ldots,A^d\}$ then
$$
|\G_D(x)|\le \sum_{i=D}^d p_i(k).
$$
\item[$(iii)$] If $\frac{1}{|X|}\sum_{x\in X}(|X|-|\G_D(x)|)=
|X| - \sum_{i=D}^d p_i(k)$ then $A_D\in\span\{I,A,\ldots,A^d\}$.
	\item[$(iv)$]
The graph $\G$ is distance-regular if and only if $D=d$ and
\begin{equation}
\label{charac-3'}
\frac{|X|}{\sum_{x\in X}\frac{1}{|X|-|\G_d(x)|}}=q_{d-1}(k)=|X|-p_d(k)=
n\left[1-\left(\sum_{i=0}^d \frac{\phi_0^2}{m_i\phi_i^2}\right)^{-1}\right]. 
\end{equation}
or, alternatively,
\begin{equation}
\label{charac-3''}
\frac{1}{|X|}\sum_{x\in X}(|X|-|\G_D(x)|)=p_d(k)=n\left(\sum_{i=0}^d \frac{\phi_0^2}{m_i\phi_i^2}\right)^{-1}. 
\end{equation}
\end{itemize}
\end{corollary}
\begin{proof}
$(i)$ Let $a_1$, $a_2$, \ldots, $a_n$ be real numbers. Recall that the numbers
$$
AM=\frac{a_1+a_2+\cdots+a_n}{n}
\qquad\mbox{ and }\qquad
HM=\frac{n}{
\frac{1}{a_1}+\frac{1}{a_2}+\cdots+\frac{1}{a_n}
}
$$
are the arithmetic and harmonic mean for the numbers $a_1$, $a_2$, \ldots, $a_n$, respectively, and we have $AM\ge HM$. Equalities occur if and only if $a_1=a_2=\cdots=a_n$. The result now follows from Theorem \ref{th:1f}.
The proofs in $(ii)$ and $(iii)$ are immediate from $(i)$, or from Theorem \ref{th:1f}.
The results in $(iv)$ correspond to different  versions of the spectral excess theorem given in \cite{vd08,f02} and \cite{fg97},  respectively. Thus, \eqref{charac-3'} is a consequence of Theorem \ref{th:1f} and Proposition \ref{charac-1}, whereas \eqref{charac-3''} follows from Theorem \ref{th:1f} and $(i)$. In these two cases, we also used $|X|=n$ and \eqref{pd(lambda0)}.
\end{proof}

\section{The Laplacian approach}

\begin{theorem}
\label{1h}
Let $\G=(X,R)$ be a connected graph with $d+1$ distinct eigenvalues, diameter $D$, and Laplacian predistance polynomials $\{r_i\}_{i=0}^d$. Then,
\begin{equation}
\label{1h}
\frac{|X|}{\sum_{x\in X}\frac{1}{|X|-|\G_D(x)|}}\ge s_{D-1}(0)=|X| - \sum_{i=D}^d r_i(0),
\end{equation}
with equality if and only if
$A_D=\sum_{i=D}^d r_i(L)$.
Moreover, in this case, if $D=2$, $\G$ is regular.
\end{theorem}

\noindent
\begin{proof}
The proof follows the same line of reasoning that in Theorem \ref{th:1f} with the polynomial $s_{D-1}$ instead of $q_{D-1}$. Thus, we have:
	\begin{eqnarray*}
		s_{D-1}(0) & = & \| s_{D-1}\|_L^2=\frac{1}{|X|}\trace( s_{D-1}(L)^2)
		=\frac{1}{|X|}\sum_{x\in X} (s_{D-1}(L)^2)_{xx}\\
		& = &
		\frac{1}{|X|}\sum_{x\in X}\sum_{y\not\in \G_D(x)} (s_{D-1}(L)_{xy})^2
		\ge
		\frac{1}{|X|}\sum_{x\in X}
		\frac{1}{|X|-|\G_D(x)|}
		\left[
		\sum_{y\not\in\G_{D}(x)}s_{D-1}(L)_{xy}
		\right]^2\\
		& = &
		\frac{1}{|X|}\sum_{x\in X}
		\frac{1}{|X|-|\G_D(x)|}
		\left[
		s_{D-1}(0)
		\right]^2,
	\end{eqnarray*}
	and this yields
	$$
	\frac{|X|}{\sum_{x\in X}\frac{1}{|X|-|\G_D(x)|}}\ge
	s_{D-1}(0)=|X|-\sum_{i=D}^d r_i(0).
	$$
	If equality holds, then for each $x$ the values of $s_{D-1}(L)_{xy}$ are the same, say $\alpha$, for all $y\not\in\G_D(x)$. Moreover, since $L$ is symmetric, $s_{D-1}(L)_{xy}=s_{D-1}(L)_{yx}$ for any $x$ and $y\not\in\G_D(x)$.
	Also, by Lemma \ref{1e}$(ii)$,
	$\sum_{y\in X} s_{D-1}(L)_{xy}=s_{D-1}(0)=n_{D-1}\alpha$, where $n_{D-1}=\sum_{i=0}^{D-1}|\G_i(x)|$ for every $x\in X$. Finally, from  $\|s_{D-1}\|^2=s_{D-1}(0)$, we have
	$$
	\|s_{D-1}\|_L^2=\frac{1}{n}\trace ( s_{D-1}(L)^2)=\frac{1}{n}\alpha^2 \suma \left(\sum_{i=0}^{D-1}A_{i}\right)=n_{D-1}\alpha^2=n_{D-1}\alpha,
	$$
	so that $\alpha=1$.
	That is, $s_{D-1}(L)_{xy}=1$ for each pair of vertices $x$ and $y$ at distance less than $D-1$. Consequently, $s_{D-1}(L)=J-A_D$, which yields $A_D=\sum_{i=D}^d r_i(L)$.
	In particular, if equality holds and $D=2$, we have $I+A=s_1(L)=r_0(L)+r_1(L)=I+r_1(L)$. Thus, $r_1(L)=A$ and, by Lemma \ref{r1-reg}, $\G$ is regular.
	
	Conversely, assume that $\sum_{i=D}^d r_i(L)=A_D$. Then, $A_D\j=\sum_{i=D}^d r_i(0)\j$, and with $|\G_D(x)|=\sum_{i=D}^d r_i(k)$, and the equality follows.
\end{proof}

From this theorem, we obtain the analogous results of Corollary \ref{1g}$(i)$-$(iv)$. In particular, the analogous of $(iv)$ yields the following characterization of distance-regularity for a (not necessarily regular) graph. 

\begin{corollary}
	\label{1i}
	Let $\G=(X,R)$ be a graph on $n$ vertices, with Laplacian matrix $L$,   Laplacian spectrum $\spec L=\{\theta_0(=0)^{m(\theta_0)},\theta_1^{m(\theta_1)},\ldots, \theta_d^{m(\theta_d)}\}$, diameter $D$, and Laplacian predistance polynomials $\{r_i\}_{i=0}^d$. Then,
	$\G$ is distance-regular if and only if $D=d$ and
		\begin{equation}
		\label{charac-3}
		\frac{|X|}{\sum_{x\in X}\frac{1}{|X|-|\G_d(x)|}}=s_{d-1}(0)=|X|-r_d(0)=n\left[1-
	\left(\sum_{i=0}^d \frac{\psi_0^2}{m(\theta_i)\psi_i^2}\right)^{-1}\right]. 
		\end{equation}
\end{corollary}
\begin{proof}
Use Theorem \ref{th:1f}, Proposition \ref{charac-2}, and \eqref{rd(0)}.
	\end{proof}

Also, as in Corollary \ref{1g}$(iv)$, the above result implies the characterization given in \cite{vf14} by using the arithmetic mean of the numbers $|X|-|\G_d(x)|$.)

\section{Open problems}

We finish the paper by formulating some open problems which could be of interest in further studies.

\begin{researchProblem}
According to Corollary \ref{1g}$(iv)$,  if $d=D$ and equality in \gzit{1f} holds, then $\G$ is distance-regular. Classify all graphs for which equality in \gzit{1f} holds.
\end{researchProblem}

\begin{researchProblem}
Let $\G=(X,R)$ be a walk-regular graph (that is, for each $\ell\ge 0$, the number of closed walks of length $\ell$ from a vertex $x$ to itself is the same for each $x$) with diameter $D$ and $d+1$ distinct eigenvalues. Assume that $D<d$. Prove or disprove that
$$
|\G_D(x)|\le p_d(k)
\qquad \forall x\in X.
$$
More generally, prove or disprove the same when $\G$ is regular.
\end{researchProblem}


\begin{researchProblem}
Let $\G=(X,R)$ be a graph with diameter $D$, adjacency matrix $A$, and $d+1$ distinct eigenvalues. Let $f_0,f_1,\ldots,f_d$ be linearly independent polynomials satisfying $\sum_{i=0}^d f_i(A)=J$, where $f_i$, for $i=0,\ldots,d$, does not need to be of degree $i$. Find under what conditions on such polynomials we can obtain a version of the spectral excess theorem for quotient polynomial graphs. (For a definition of quotient polynomial graphs, see \cite{f16}).
\end{researchProblem}

\newpage


\end{document}